\documentclass[10pt]{amsart}
\usepackage[T1]{fontenc}
\usepackage[latin1]{inputenc}
\usepackage[english]{babel}
\usepackage{amsmath}
\usepackage{amssymb}
\usepackage{amsthm}
\usepackage{amsfonts}

\theoremstyle{plain}
\newtheorem{lemma}{Lemma}
\newtheorem{theorem}[lemma]{Theorem}
\newtheorem{proposition}[lemma]{Proposition}
\newtheorem{remark}[lemma]{Remark}

\newcommand{\bean}{\begin{eqnarray*}}
\newcommand{\eean}{\end{eqnarray*}}
\newcommand{\Nbean}{\begin{eqnarray}}
\newcommand{\Neean}{\end{eqnarray}}
\newcommand{\ba}{\begin{array}}
\newcommand{\ea}{\end{array}}
\newcommand{\R}{\mathbb R}
\newcommand{\N}{\mathbb N}

\newcommand{\I}{1 \kern-0.25em{\rm l}} 
\newcommand{\twodots}{..\kern0.0em}

\newcommand{\Beta}{\text{\rm Beta}}
\def\wrt{w.\,r.\,t.\ }

\title[Stability of the optimal filter in a HMM with multiplicative noise]{Stability of the optimal filter in a hidden Markov model with multiplicative noise}

\author{Birgit Debrabant}
\author{Wilhelm Stannat}


\begin{document}
\maketitle
\begin{abstract}
We consider a hidden Markov model with multiplicative noise emerging from studies of software reliability.
We show the stability of the optimal filter \wrt general initial conditions in the total variation- and $L^p$-norm and deduce explicit rates.
Remarkably, stability turns out to be independent of the ergodic behavior of the signal.
\end{abstract}

\section{Introduction}
The stability of nonlinear filters is a field of active research, see e.\,g.\ \cite{Oxford}, the introduction of \cite{vanHandel} and references therein.
However, the majority of results requires the signal process to be ergodic or stable in some sense.
In addition, most of the results are obtained for signals observed with additive noise. The case of nonergodic signals and also the case of signals observed with multiplicative noise still remains mostly open. 
In this context,
the present article studies the stability of the optimal filter in the following hidden Markov model:
\begin{eqnarray}
\text{Signal: } X_{n} &=& b X_{n-1} W_{n},\label{Eq:Signal}\\
\text{Observation: } Y_n &=& X_n^{-1} G_n, \qquad n \in \N,
\label{Eq:Beobachtung}
\end{eqnarray}
where $W_n$, $n \in \N$, are independent identically $\Beta(\alpha,\beta)$-distributed random variables, $\alpha,\beta>0$, describing the noise incorporated in the unknown signal, $b$ is a positive parameter depending on which the unknown signal process can be ergodic or nonergodic and where $G_n$, $n \in \N$, are independent $\Gamma(1,\beta)$-distributed.
Hence, the observation $Y_n$ depends on $X_n$ via multiplication with the independent noise $G_n$.
Thus, model \eqref{Eq:Signal} and \eqref{Eq:Beobachtung} is an example of filtering a signal observed with
multiplicative noise.
Note that although to logarithmize leads to a classical linear model with additive noise, stability cannot be studied immediately with known methods such as proposed in e.\,g.\ \cite{vanHandel2, DoucMoulinesRitov} since the corresponding noise terms are rather irregular e.\,g.\ neither unimodal nor do they have light tails.

Essentially, the above model appears in \cite{Bather} as example for models admitting explicit invariant conditional distributions. In our case this amounts to the fact that the incorporated assumptions on the distributions of signal and observation together with a corresponding initial distribution, i.\,e.\\
\begin{tabular}{ll}
signal:& $(X_n/bX_{n-1}|X_{n-1}) \sim \Beta(\alpha,\beta)$\\
initial distribution:& $X_0\sim \Gamma(\lambda,q)$, where $q=\alpha+\beta$ and $\lambda>0$\\
observation:& $(Y_n|X_n) \sim \Gamma(X_n, \beta)$\\
\end{tabular}~\\
imply the following explicit updating rules:

\begin{tabular}{ll}
posterior of $X_{n}$:& $(X_n|Y_{1:n}) \sim \Gamma(\lambda_{n},q)$, where $Y_{1:n}=Y_1, \ldots, Y_{n}$\\
prior of $X_{n+1}$:& $(X_{n+1}|Y_{1:n}) \sim \Gamma(\lambda_n/b, \alpha)$\\
1-step ahead prediction:& $(b X_{n+1}/\lambda_n|Y_{1:n}) \sim \mbox{Pearson Type VI}$\\& with parameters $\beta$ and $\alpha$, cp.\ \cite{JohnsonKotz}\\
posterior of $X_{n+1}$:& $(X_{n+1}|Y_{1:n+1}) \sim \Gamma(\lambda_{n+1},q)$,\\& where $\lambda_{n+1}=\frac{\lambda_n}b+ y_{n+1}$
\end{tabular}

To study software reliability, model \eqref{Eq:Signal} and \eqref{Eq:Beobachtung} was later applied in \cite{SingpurwallaWilson, Singpurwalla} as enhancement to the Kalman filter taking into account that failure data tends to be highly skewed and observational errors are not mainly caused by instrumental inaccuracies. Thereby, the observables $Y_n$ can be interpreted as interfailure times of some software.
The $X_n$ play the role of unknown parameters steering their distribution.
To model an evolution of the software the parameters evolve according to \eqref{Eq:Signal}. The value of $b$ is typically unknown and indicates if we have a tendency of increasing reliability since e.\ g.\
$E(Y_n|Y_{1:n-1},X_0) \;=\; 2 b^{-n} \left(X_0+\sum_{j=1}^{n-1} b^j Y_j \right)$
obviously tends to infinity if $b< 1$.

Our stability results give the dependence of the optimal filter $\pi_n^{y_{1:n}}$, that is the regular conditional distribution of $X_n$ given the observations $y_{1:n}=(y_k)_{k=1, \ldots, n}$, on the initial distribution $\pi_0$ of $X_0$.
To cover a wider range of admissible initial conditions we extend the assumptions of \cite{Bather, Singpurwalla} and suppose the initial distribution of $X_0$ to be a mixed Gamma-distribution with the following density
\begin{equation}
\label{IniCond}
\pi_0(x) \;\propto\; x^{q-1} \int_0^\infty \lambda^q e^{-\lambda x} d U_0(\lambda), \qquad x>0,
\end{equation}
where $q=\alpha+\beta$ and $U_0$ is a probability measure on a compact subset of $(0,\infty)$. Such a mixture conserves the conjugacy of this distribution, cp.\ Lemma \ref{conj} which shows that all the posterior distributions are of a similar type.

The pure Gamma case, where $U_0$ is a Dirac measure is easy to treat and will serve us as an introductory example, see Section \ref{Sec:pure}.
Provided that the unknown true initial condition $\tilde\pi_0$ is absolutely continuous \wrt the assumed initial condition $\pi_0$ with a density bounded away from 0,
we show stability in the total variation norm (almost surely given the observations) and in the $L_p$-norm for arbitrary $p>0$ (expectation \wrt the observations $Y_1, Y_2,\ldots$) both with explicit geometric rates.

In Sections \ref{Sec:mixed} and \ref{Sec:mixed2} we pass on to the general mixed Gamma initial condition. Under assumptions similar to those in Sec.\ \ref{Sec:pure} we show stability with geometric rates.
Concerning the total variation norm, Theorem \ref{Thm:stability} gives that almost surely
\begin{equation}
\label{Eq:Rate}
|| \pi_n^{y_{1:n}}-\tilde \pi_n^{y_{1:n}}||_{var} \;=\; \mathcal O (\delta^{-n}), \qquad n \to \infty,
\end{equation}
for $\delta<E(W_1^\beta)^{-\frac1\beta}$, which coincides with the pure Gamma case.
Note that the rate of stability $\frac 1 \delta$ is independent of the parameter $b$ in the signal. That is, the filter is stable whether or not the signal is ergodic. However, the constant on the right hand side of \eqref{Eq:Rate} will depend on the given sequence $y_1, y_2, \ldots$ of observations.
For the $L_p$-norm with $p \in \left(0,-\frac{u_0}{\bar B}\ln E( W_1^\beta ) \right)$,
Theorem \ref{Thm:Lp} yields
\[E\left(|| \pi_n^{Y_{1:n}}-\tilde \pi_n^{Y_{1:n}}||_{var}^p\right) =\mathcal O(\rho^n), \qquad n\to \infty,\]
where
$\rho=\left(E ( W_1^\beta ) e^{\frac{p \bar B}{u_0}} \right)^{\frac p{p+\beta}}$ and where $\bar B=\bar B(\beta, u_0, o_0)$ is a positive constant specified in the theorem. These rates are smaller compared to $\rho=E(W_1^p)$ in the pure Gamma case.

All stability statements of this article are based on an universal stability result of \cite{Stannat}.
Therein, a general bound for the total variation distance of optimal filters \wrt to the initial condition is deduced in terms of the Lipschitz contraction $\chi_n^*$ of a certain parabolic ground state transform $P^*_n$ associated with $\pi_n^{y_{1:n}}$.
More precisely, suppose that the optimal filter is erroneously initialized
with initial condition $\pi_0$ and that $\tilde\pi_0$ is the true initial condition. If $\tilde\pi_0 \ll\pi_0$ with density $h_0=\frac{d \tilde \pi_0}{d\pi_0}$ it follows that $\tilde\pi_n^{y_{1:n}} \ll \pi_n^{y_{1:n}}$ with density $h_n$ given by $h_n=P_n^* \circ \cdots \circ P_1^* h_0$, $n \in \N$, where $P_n^* f:=\frac{g(\cdot, y_n) \hat P(f  \pi_{n-1}^{y_{1:n-1}})}{\pi_n^{y_{1:n}} \int g(s, y_n)\hat P \pi_{n-1}^{y_{1:n-1}}(s)ds}$, $\hat P$ resp.\ $\hat p$ is the dual operator of the transition probability $P$ resp.\ the transition density $p$ of the signal, that is $P(X_{n+1} \in dx_{n+1}|X_n)=p(X_n,x_{n+1})dx_{n+1}$ and
$\hat P f=\int p(x,\cdot) f(x)dx$, and where $g$ denotes the regular conditional density of the observation given the signal, i.\,e.\ $P(Y_n\in dy|X_n)=g(X_n,y)dy$.
It follows that the error between the true optimal filter
$\tilde \pi_n^{y_{1:n}}$ and the erroneously initialized filter $\pi_n^{y_{1:n}}$ can be expressed in
terms of the Lipschitz contractions $\chi_n^*$ of the Markovian transition kernels $P_n^*$,
whereby
$\displaystyle\chi_n^*:=\sup_{f \in Lip}\frac{||P_n^*f||_{Lip}}{||f||_{Lip}}$ and where $Lip$ is the space of all Lipschitz continuous functions with the corresponding norm $||\cdot||_{Lip}$
:

\begin{proposition}[cp.\ Prop.\ 2.1 of \cite{Stannat}]
\label{Prop:Stannat}
For the total variation distance of the true optimal filter $\tilde\pi_n^{y_{1:n}}$ to a wrongly initialized $ \pi_n^{y_{1:n}}$ the following explicit bound is valid:
\begin{equation}
|| \pi_n^{y_{1:n}}-\tilde \pi_n^{y_{1:n}}||_{var}
\;\le \;
\frac{\hat \sigma_n}{2 H}
\cdot
\prod_{k=1}^n \chi_k^* \cdot ||h_0||_{Lip}, \qquad n \in \N,
\end{equation}
provided that $\tilde \pi_0$ has $\pi_0$-density $h_0$ which is bounded away from 0 by a positive constant $H$
and with
$\hat \sigma_n
=
\int\limits_0^\infty \!\! \int\limits_0^\infty |x_1-x_2|\, d\pi_n^{y_{1:n}}(x_1)d\pi_n^{y_{1:n}}(x_2)
$.
\end{proposition}

Note that the formulation in \cite{Stannat} is more general and gives a bound also in the case $H=0$.

Clearly, for our model \eqref{Eq:Signal} and \eqref{Eq:Beobachtung} we obtain
\begin{eqnarray*}
p(x,y)=\frac{(bx)^{1-q}}{\Beta(\alpha,\beta)} y^{\alpha-1} (bx-y)^{\beta-1}\I_{[0,bx]}(y)\qquad\text{and}\\
\hat p(x,dy)=p(y,x)dy \linebreak[1]=\frac{(by)^{1-q}}{\Beta(\alpha,\beta)}x^{\alpha-1}(by-x)^{\beta-1} \I_{[b^{-1}x,\infty)}(y)dy.
\end{eqnarray*}

\section{Stability \wrt pure Gamma initial conditions}
\label{Sec:pure}

We consider the particular case
$\pi_0(x) \propto x^{q-1}\lambda^q e^{-\lambda x}$, $x>0$,
for some $\lambda>0$, hence $X_0$ is $\Gamma(\lambda,q)$-distributed.
Using the recursion
\begin{eqnarray}
\label{Eq:posterior}
\pi_{n+1}^{y_{1:n+1}}(x)
&\propto&
g(x,y_{n+1}) \int p(s,x) \pi_n^{y_{1:n}}(ds),
\end{eqnarray}
it is straightforward to show that the corresponding optimal filters $\pi_n^{y_{1:n}}$, $n \in \N$, are again Gamma-distributed, namely $\Gamma(\lambda_n,q)$ with  $\lambda_n=b^{-n}\lambda+ \sum_{j=1}^n b^{j-n} y_j$.

To study the asymptotic dependence of the optimal filter on the initial condition we apply Prop.\ \ref{Prop:Stannat}. Since
\begin{equation}
p_{n}^*(x,dy) \;=\; \frac{\lambda_{n-1}^\beta}{\Gamma(\beta)} (y-b^{-1}x)^{\beta-1} e^{-\lambda_{n-1}(y-b^{-1}x)} \I_{(b^{-1}x,\infty)}(y) dy,
\end{equation}

we find $||P_{n}^*||_{Lip}=b^{-1}$, $n \in \N$.
Further
\begin{eqnarray*}
\hat \sigma_n
&\le&
\left(
\int\limits_0^\infty \!\! \int\limits_0^\infty
(x_1-x_2)^2
d\Gamma(\lambda_n,q)(x_1)d\Gamma(\lambda_n,q)(x_2)
\right)^\frac12
\;=\;
\frac{\sqrt{2q}}{\lambda_n}, \qquad n\in \N.
\end{eqnarray*}

Prop.\ \ref{Prop:Stannat} then implies
\begin{equation}
\label{Eq:pure}
|| \pi_n^{y_{1:n}}-\tilde \pi_n^{y_{1:n}}||_{var}
\;\le \;
\frac{\sqrt{2q}||h_0||_{Lip}}{2 H}
(b^n \lambda_n)^{-1}
, \qquad n\in \N,
\end{equation}
if the true initial condition $\tilde \pi_0$ has a $\pi_0$-density which is bounded away from 0 by some constant $H>0$.

The following lemma analyses the limiting behavior of $b^n\lambda_n=\lambda+\sum_{j=1}^n b^j y_j$ which evidently plays a crucial role concerning stability:

\begin{lemma}
\label{div}
$P(\liminf_{j \to \infty} \{b^jY_j>\delta^j\})=1$ for arbitrary $\delta<E(W_1^\beta)^{-\frac1\beta}$.
In particular, $\sum_{j=1}^\infty b^j Y_j=+\infty$ almost surely.
\end{lemma}

\begin{proof}
We may assume $\delta\ge 0$.
Note that
\begin{eqnarray}
\label{Eq:bound}
P(b^jY_j\le \delta^j) &=& E\left(\frac{X_j^\beta}{\Gamma(\beta)} \int_0^{\delta^j b^{-j}}y^{\beta-1} e^{-X_jy}dy \right)\nonumber\\
&\le&
\frac{\delta^{j\beta}}{\Gamma(\beta+1)}\cdot \underbrace{E\left((b^{-j}X_j)^\beta\right) }_{=\left(E(W_1^\beta)\right)^j E\left( X_0^\beta \right)}
\end{eqnarray}
is summable for
$\delta < E\left( W_1^\beta \right)^{-\frac1\beta}$.
Hence, the Borel-Cantelli lemma yields\linebreak[4] $P(\liminf_{j \to \infty} \{b^jY_j>\delta^j\})=1$.
Since $E(W_1^\beta)^{-\frac1\beta}>1$, we can choose $\delta>1$ which implies that $\sum_{j=1}^\infty b^j Y_j$ diverges to infinity almost surely.
\end{proof}

\begin{remark}
\label{Rm:limit}
Lemma \ref{div} remains valid for any initial condition $\pi_0$ satisfying \mbox{$E(X_0^\beta)<\infty$.}
\end{remark}

Concerning almost sure stability, Lemma \ref{div} implies that \eqref{Eq:pure} converges to 0 almost surely and behaves as $\mathcal O(\delta^{-n})$ for every $\delta < E ( W_1^\beta )^{-\frac 1 \beta}$.

Moreover, concerning stability in the $L^p$-norm, for $0<p<\beta$ we find
\begin{eqnarray*}
E((b^n Y_n)^{-p})
&=&
E\left(
\int_0^\infty (b^n y)^{-p} \frac{X_n^\beta}{\Gamma(\beta)} y^{\beta-1} e^{-X_n y} dy\right)\\
&=&
\frac{\Gamma(\beta-p)}{\Gamma(\beta)}
E \left(
(b^{-n}X_n)^p
\right)
\;=\;
\frac{\Gamma(\beta-p)E(X_0^p)}{\Gamma(\beta)} E\left(W_1^p\right)^n,
\end{eqnarray*}
$n \in \N$. Therefore, the optimal filter is stable in the corresponding $L^p$-norm and satisfies
\begin{eqnarray*}
E\left(|| \pi_n^{Y_{1:n}}-\tilde \pi_n^{Y_{1:n}}||_{var}^p\right)
&=& \mathcal O\left( \rho^n \right), \qquad n \to \infty,
\end{eqnarray*}
with $\rho=E(W_1^p)=\frac{\Beta(\alpha+p,\beta)}{\Beta(\alpha,\beta)}$.

\section{Stability in the total variation norm}
\label{Sec:mixed}
We return to the more general mixed Gamma initial conditions of the form \eqref{IniCond} and deduce an estimate of the total variation distance of $\pi_n^{y_{1:n}}$ \wrt differing initial conditions in Thm.\ \ref{stability} which then implies almost sure stability with geometric rates, cp.\ Thm.\ \ref{Thm:stability}.

\begin{theorem}
\label{stability}
Let $\pi_0$ and $\tilde \pi_0$ be such that
\begin{itemize}
\item
$\pi_0(x) \propto x^{q-1} \int_0^\infty \lambda^q e^{-\lambda x} d U_0(\lambda)$ for some probability measure $U_0$ with support $[u_0,o_0] \subset (0,\infty)$, and
\item
$\tilde \pi_0$ is a probability density on $\R_+$ such that the corresponding density $h_0 :=\frac{\tilde\pi_0}{\pi_0}$ is Lipschitz continuous and $h_0 \ge H>0$ for some constant $H$.
\end{itemize}

Let $y_{1:n} \in \R_+^n$ and let $\pi_n^{y_{1:n}}$ resp.\ $\tilde \pi_n^{y_{1:n}}$ be the optimal filter \wrt $\pi_0$ resp.\ $\tilde \pi_0$.
Then
\begin{equation}
\label{Eq:bound2}
|| \pi_n^{y_{1:n}}-\tilde \pi_n^{y_{1:n}}||_{var}
\;\le \;
\frac{Q}{b^{n}u_n} \prod_{k=0}^{n-1} \left(1+B \frac{o_0-u_0}{b^{k}u_k} \right)||h_0||_{Lip}
, \qquad n \in \N,
\end{equation}

where $Q=\frac{\sqrt{q(q+1)}}{\sqrt 2H}$ and $B=\frac12\left(
\beta(\beta+1)
-\frac{2\beta^2 u_0}{o_0}
+\frac{\beta(\beta+1)o_0^2}{u_0^2}
\right)^\frac12
$.
\end{theorem}

This estimate induces stability with geometric rates:

\begin{theorem}
\label{Thm:stability}
Under the assumptions of Theorem \ref{stability} the optimal filter is stable for almost all sequences $y_1, y_2, \ldots$ of observations and we have
\begin{equation}
\label{Eq:stab}
|| \pi_n^{y_{1:n}}-\tilde \pi_n^{y_{1:n}}||_{var} \;=\; \mathcal O (\delta^{-n}), \qquad n \to \infty,
\end{equation}
for every $\delta<E(W_1^\beta)^{-\frac1\beta}$.
\end{theorem}

\subsubsection*{Proofs and preliminary results}
Before proving the above main statements we consider two preliminary results. Firstly note that posterior distributions corresponding to an initial condition of the form \eqref{IniCond} are again of mixed Gamma type:

\begin{lemma}
\label{conj}
The posterior distributions $\pi_n^{y_{1:n}}$ are mixed Gamma-distributions with densities
\[\pi_n^{y_{1:n}}(x) \propto x^{q-1} \int_0^\infty \lambda^q e^{-\lambda x} d U_n(\lambda), \qquad x>0.\]
Hereby, the distributions $U_n$, $n \in \N$, obey the following recursive scheme:
\begin{itemize}
\item $U_n^{(1)}(d\lambda)= \lambda^\alpha U_n(d\lambda)$,
\item $U_n^{(2)}(A) = U_n^{(1)}(b \cdot A)$ for $A \in \mathcal B(\R_+)$,
\item $U_n^{(3)}= \delta_{y_{n+1}} \star U_n^{(2)}$, where $\delta_{y_{n+1}}$ denotes the Dirac measure in $y_{n+1}$ and $\star$ the convolution,
\item $U_{n+1}(d\lambda)= \lambda^{-q} U_n^{(3)}(d\lambda)$.
\end{itemize}
\end{lemma}

\begin{remark}
\label{rem1}
For $U_0=\delta_{\lambda}$ with some $\lambda>0$ we obtain $U_n=\delta_{b^{-n} \lambda+\sum_{k=1}^n b^{k-n}y_k}$. This setup corresponds to the pure Gamma case of Section \ref{Sec:pure}.

If $U_0$ has compact support $[u_0,o_0]$ then $U_n$ has also a compact support $[u_n,o_n]$ with
$u_n=b^{-n} u_0+\sum_{k=1}^n b^{k-n} y_k$ and $o_n=b^{-n} o_0+\sum_{k=1}^n b^{k-n} y_k$, $n \in \N$.
\end{remark}

\begin{proof}
The recursion \eqref{Eq:posterior} implies
\begin{eqnarray*}
\lefteqn{\pi_{n+1}^{y_{1:n+1}}(dx)}\\
&\propto&
x^\beta e^{-x y_{n+1}} \int_{b^{-1}x}^\infty (bs)^{1-q} x^{\alpha-1} (bs-x)^{\beta-1} \cdot s^{q-1} \int_0^\infty \lambda^q e^{-\lambda s} dU_n(\lambda) ds\\
&\propto&
x^{q-1} e^{-x y_{n+1}} \int_0^\infty \lambda^\alpha e^{-\lambda b^{-1}x} dU_n(\lambda)\\
&=&
x^{q-1} e^{-x y_{n+1}} \int_0^\infty e^{-\lambda b^{-1}x} dU_n^{(1)}(\lambda)
\;=\;
x^{q-1} e^{-x y_{n+1}} \int_0^\infty e^{-\lambda x} dU_n^{(2)}(\lambda)\\
&=&
x^{q-1} \int_0^\infty e^{-\lambda x} dU_n^{(3)}(\lambda)
\;=\;
 x^{q-1} \int_0^\infty \lambda^q e^{-\lambda x} d U_{n+1}(\lambda), \qquad x>0.
\end{eqnarray*}

\end{proof}

In order to apply Prop. \ref{Prop:Stannat} we deduce an upper bound for the contraction coefficient $\chi_n^*$ of the ground state transform $P_n^*$:

\begin{proposition}
\label{contraction}
Let $f:\R_+\to \R$ be Lipschitz-continuous.
Under the assumptions of Thm. \ref{stability} we find
\[||P_n^*f||_{Lip} \le  \frac{\frac12\sup_{x> 0} d_n(x)+1}{b} \,||f||_{Lip}\]
with
\[
d_n(x)
\;=\;
\int\limits_0^\infty\!\! \int\limits_0^\infty
|\lambda_1 -\lambda_2|
\left(
\frac{\beta(\beta+1)}{\lambda_1^2}
-\frac{2\beta^2}{\lambda_1 \lambda_2}
+\frac{\beta(\beta+1)}{\lambda_2^2}
\right)^\frac12
dU_{n-1}^x(\lambda_1) dU_{n-1}^x(\lambda_2)
\]
and
\[d  U_{n-1}^x(\lambda)
=
\frac{\lambda^\alpha e^{-\lambda b^{-1}x}dU_{n-1}(\lambda)}
{\int\lambda^\alpha e^{-\lambda b^{-1}x} dU_{n-1}(\lambda)}, \qquad n \in \N, \quad x > 0.
\]
Moreover
\[\sup_{x>0}d_n(x)\;\le\; \left(\frac{o_{n-1}}{u_{n-1}}-1\right)
\left(
\beta(\beta+1)
-\frac{2\beta^2 u_0}{o_0}
+\frac{\beta(\beta+1)o_0^2}{u_0^2}
\right)^\frac12, \qquad n\in \N.
\]

\end{proposition}

\paragraph{Notations:}
To simplify notations in the following proof we introduce for $x,y>0$ and $n \in \N$ the following measures derived from $U_n$:
\begin{eqnarray*}
dU_n^x(\lambda)
&=&
 \frac{\lambda^\alpha e^{-\lambda b^{-1}x } dU_n(\lambda) }{\int_0^\infty \lambda^\alpha e^{-\lambda b^{-1}x } dU_n(\lambda) },\\
dU_n^{x,y}(\lambda)
&=&
\frac{\lambda^\beta}{\Gamma(\beta)}y^{\beta-1}e^{-\lambda y} dU_n^x(\lambda),\\
d\bar U_n(\lambda,y)
&=&
dU_n^{x,y}(\lambda)dy.
\end{eqnarray*}

Moreover, note the following estimate via Jensen's inequality
\begin{eqnarray}
\lefteqn{\int\limits_0^\infty\!\! \int\limits_0^\infty
|y_1-y_2| d\Gamma(\lambda_1,\beta)(y_1) d\Gamma(\lambda_2,\beta) (y_2)}\nonumber\\
&\le&
\left(
\int\limits_0^\infty\!\! \int\limits_0^\infty
(y_1-y_2)^2 d\Gamma(\lambda_1,\beta)(y_1) d\Gamma(\lambda_2,\beta) (y_2)
\right)^\frac12 \nonumber\\
&=&
\left(
\frac{\beta(\beta+1)}{\lambda_1^2}
-\frac{2\beta^2}{\lambda_1 \lambda_2}
+\frac{\beta(\beta+1)}{\lambda_2^2}
\right)^\frac12.
\label{Eq:Formel}
\end{eqnarray}

\begin{proof}
Let $n \in \N$.
Due to the structure of the optimal filter we obtain for $x,y>0$
\begin{eqnarray*}
p_n^*\left(x,y\right)
&=&
\frac{
\int_0^\infty e^{-\lambda y} (y-b^{-1}x)^{\beta-1}\lambda^q dU_{n-1}(\lambda) }
{\Gamma(\beta) \int_0^\infty \lambda^\alpha e^{-\lambda b^{-1}x}dU_{n-1}(\lambda) }
\I_{(b^{-1}x,\infty)}(y).
\end{eqnarray*}

For $x_1, x_2 >0$ we find
\begin{eqnarray*}
\lefteqn{P_n^*f(x_1)-P_n^*f(x_2)}\\
&=&
\underbrace{
\int_0^\infty
\left[f\left(y+\frac{x_1}b\right) -f\left(y+\frac{x_2}b\right)\right] p_n^*\left(x_2,y+\frac{x_2}b\right) dy}_{=:T_1}\\
&&+
\underbrace{
\int_0^\infty
f\left(y+\frac{x_2}b\right)
\left[p_n^*\left(x_1,y+\frac{x_1}b\right)- p_n^*\left(x_2,y+\frac{x_2}b\right)\right] dy}_{=:T_2}.
\end{eqnarray*}

As $x \mapsto p_n^*(x,y+b^{-1}x)$ is differentiable with
\begin{eqnarray*}
\lefteqn{\frac{d}{dx} p_n^*\left(x,y+\frac{x}b\right)}\\
&=&
\frac{
\int_0^\infty -\lambda b^{-1} e^{-\lambda(y +b^{-1}x)}y^{\beta-1}\lambda^q dU_{n-1}(\lambda) }
{\int_0^\infty \lambda^\alpha e^{-\lambda b^{-1}x}dU_{n-1}(\lambda) \Gamma(\beta)}\\
&&
+
\frac{
\int_0^\infty e^{-\lambda(y +b^{-1}x)}y^{\beta-1}\lambda^q dU_{n-1}(\lambda) \int_0^\infty \lambda b^{-1} \lambda^\alpha e^{-\lambda b^{-1}x}dU_{n-1}(\lambda)  }
{\left( \int_0^\infty \lambda^\alpha e^{-\lambda b^{-1}x}dU_{n-1}(\lambda)\right)^2 \Gamma(\beta) }\\
&=&
-\int_0^\infty \frac\lambda b dU_{n-1}^{x,y}(\lambda)
+
\int_0^\infty dU_{n-1}^{x,y}(\lambda)
\cdot
\int_0^\infty \!\!\int_0^\infty \frac\lambda b
dU_{n-1}^{x,y}(\lambda) dy,
\end{eqnarray*}

for $T_2$ we find
\begin{eqnarray*}
T_2
&=&
\int_0^\infty
f\left(y+\frac{x_2}b\right)
\int_{x_1}^{x_2} \frac{d}{dx} p_n^*\left(x,y+\frac{x}b\right) dx dy\\
&=&
\int_{x_1}^{x_2}
\underbrace{
\int_0^\infty
f\left(y+\frac{x_2}b\right)
\frac{d}{dx} p_n^*\left(x,y+\frac{x}b\right)dy}_{=:T_3}
 dx.
\end{eqnarray*}

Further we have
\begin{eqnarray*}
T_3
&=&
\int\limits_0^\infty \!\! \int\limits_0^\infty
f\left(y+\frac {x_2}b\right) \frac\lambda b d\bar U_{n-1}^x(\lambda,y)\\
&&
-
\int\limits_0^\infty \!\! \int\limits_0^\infty
f\left(y+\frac {x_2}b\right)
d\bar U_{n-1}^x(\lambda,y)
\cdot
\int\limits_0^\infty \!\! \int\limits_0^\infty \frac\lambda b
d\bar U_{n-1}^x(\lambda,y)\\
&=&
\frac12
\int\limits_0^\infty\!\! \int\limits_0^\infty\!\! \int\limits_0^\infty\!\! \int\limits_0^\infty
\left(f\left(y_1+\frac {x_2}b\right)-f\left(y_2+\frac {x_2}b\right)\right)\left( \frac{\lambda_1} b -\frac{\lambda_2} b \right)\\
&&\hspace{5cm}
d\bar U_{n-1}^x(\lambda_1,y_1) d\bar U_{n-1}^x(\lambda_2,y_2),
\end{eqnarray*}
since $\bar U_{n-1}^x$ is a probability measure on $[0,\infty)^2$.
Therefore,
\begin{eqnarray*}
|T_3|
&\le&
\frac{||f||_{Lip}}{2 b}
\int\limits_0^\infty\!\! \int\limits_0^\infty\!\! \int\limits_0^\infty\!\! \int\limits_0^\infty
|y_1-y_2| \cdot |\lambda_1 -\lambda_2|\,
d\bar U_{n-1}^x(\lambda_1,y_1) d\bar U_{n-1}^x(\lambda_2,y_2)\\
&=&
\frac{||f||_{Lip}}{2 b}
\int\limits_0^\infty\!\! \int\limits_0^\infty\!\!
\int\limits_0^\infty\!\! \int\limits_0^\infty
|y_1-y_2| \cdot
\frac{(\lambda_1 \lambda_2)^\beta}{\Gamma(\beta)^2}(y_1 y_2)^{\beta-1}e^{-\lambda_1 y_1-\lambda_2 y_2} dy_1 dy_2\\
&& \hspace{5cm} \cdot
|\lambda_1 -\lambda_2|\,
dU_{n-1}^x(\lambda_1) dU_{n-1}^x(\lambda_2)
\end{eqnarray*}
and formula \eqref{Eq:Formel} yields

\begin{eqnarray*}
|T_3|
&\le&
\frac{||f||_{Lip}}{2 b}
\int\limits_0^\infty\!\! \int\limits_0^\infty
\frac{|\lambda_1 -\lambda_2|}{\lambda_1}
\left(
\beta(\beta+1)
-\frac{2\beta^2\lambda_1}{\lambda_2}
+\frac{\beta(\beta+1)\lambda_1^2}{\lambda_2^2}
\right)^\frac12\\
&&
\hspace{7cm}
dU_{n-1}^x(\lambda_1) dU_{n-1}^x(\lambda_2)\\
&\le&
\frac{||f||_{Lip}}{2 b}
\left(\frac{o_{n-1}}{u_{n-1}}-1\right)
\underbrace{
\left(
\beta(\beta+1)
-\frac{2\beta^2 u_0}{o_0}
+\frac{\beta(\beta+1)o_0^2}{u_0^2}
\right)^\frac12
}_{=:2 B}
\end{eqnarray*}
since
$1 \le \frac{o_n}{u_n} \le \frac{o_0}{u_0}$
and $\frac{o_n}{u_n} \downarrow 1$ due to Lemma \ref{div}.
For $T_2$ we finally have
\begin{eqnarray*}
|T_2| &\le&
|x_1-x_2| \frac{||f||_{Lip}}{b}
\left(\frac{o_{n-1}}{u_{n-1}}-1\right)
B
\end{eqnarray*}
which together with the obvious estimate
$|T_1| \le \frac{||f||_{Lip}}{b}\cdot |x_1-x_2|$
yields
\begin{eqnarray*}
||P_n f||_{Lip}
&\le&
\frac{1+\left(\frac{o_{n-1}}{u_{n-1}}-1\right)B}{b}
||f||_{Lip}\,.
\end{eqnarray*}

\end{proof}

Now we can proceed with the proofs of Theorems \ref{stability} and \ref{Thm:stability}:

\begin{proof}
of Thm. \ref{stability}:
We apply Prop.\ \ref{Prop:Stannat}.
Due to Prop.\ \ref{contraction} we have
\begin{eqnarray*}
\chi_k^*
&\le&
b^{-1} \left(1+B\frac{o_{k-1}-u_{k-1}}{u_{k-1}} \right)
\;=\;
b^{-1} \left(1+B \frac{o_0-u_0}{b^{k-1}u_{k-1}} \right).
\end{eqnarray*}

For $\hat \sigma_n$ we find
\begin{eqnarray*}
\hat \sigma_n
&=&
%
\int\limits_0^\infty \!\! \int\limits_0^\infty \!\! \int\limits_0^\infty \!\! \int\limits_0^\infty
|x_1-x_2| d\Gamma(\lambda_1,q)(x_1)d\Gamma(\lambda_2,q)(x_2)
dU_n(\lambda_1)dU_n(\lambda_2)\\
&\le&
\left(
\int\limits_0^\infty \!\! \int\limits_0^\infty
\frac{q(q+1)}{\lambda_1^2}-2\frac{q^2}{\lambda_1 \lambda_2}+\frac{q(q+1)}{\lambda_2^2}
dU_n(\lambda_1)dU_n(\lambda_2)
\right)^\frac12\\
&\le&
%
u_n^{-1}
(2 q(q+1))^\frac12
\end{eqnarray*}
due to \eqref{Eq:Formel}. Altogether this yields \eqref{Eq:bound2} and proves the theorem.
\end{proof}

\begin{proof} of Thm. \ref{Thm:stability}:\\
We show that $(b^{n} u_n)^{-1} \prod_{k=0}^{n-1} \left(1 +B \frac{o_0-u_0}{b^{k}u_k} \right)$ vanishes almost surely if $n$ goes to infinity:

First note that
\begin{eqnarray*}
\prod_{k=0}^{n-1} \left(1 +B \frac{o_0-u_0}{b^{k}u_k} \right)
&\le&
\exp\left\{ B (o_0-u_0) \sum_{k=0}^{n-1} (b^{k} u_k)^{-1} \right\},
\end{eqnarray*}
since $1+x \le \exp x$ for $x \in \R$.

Remember that $b^ku_k=u_0+\sum_{j=1}^k b^j y_j$.
Due to Lemma \ref{div} we can find $\delta>1$ and a random index $J_\delta$ a.\,s.\ finite such that $Y_j>(\delta b^{-1})^j$ for all $j \ge J_\delta$. Consequently, there is a $j_\delta \in \N$ with
\begin{eqnarray}
\frac{
\prod_{k=0}^{n-1} \left(1 +B \frac{o_0-u_0}{b^{k}u_k} \right)
}{b^n u_n}
%
%
&\le&
\frac{\exp\left\{ \bar B \sum_{k=0}^{j_\delta-1} \frac1{u_0} + \bar B \sum_{k=j_\delta}^{n-1} \frac1{\delta^k } \right\}}
{u_0+\sum_{j=j_\delta}^n \delta^j},
\label{rechts}
\end{eqnarray}
where $\bar B=B(o_0-u_0)$.
The right hand side of \eqref{rechts} can be bounded for $n\ge j_\delta$ by
\begin{eqnarray}
\frac{\exp\left\{\frac{\bar B}{u_0} j_\delta\right\}
\exp \left\{\bar B \frac{\delta}{\delta-1} \right\}}
{\delta^n}
,
\label{Jestimate}
\end{eqnarray}
which tends to 0 as $\mathcal O (\delta^{-n})$ for $n \to \infty$.
\end{proof}

\section{Stability in the $L^p$-norm}
\label{Sec:mixed2}
In the setting of Theorem \ref{stability} the optimal filter is always stable in the $L^p$-norm for every $p>0$ since the total variation norm is bounded by 1 and so almost sure convergence of the optimal filter implies

\[E\left(|| \pi_n^{Y_{1:n}}-\tilde \pi_n^{Y_{1:n}}||_{var}^p\right) \to 0, \qquad n\to \infty,\]
by dominated convergence.

Moreover, for some values of $p$ we can specify the rates of this convergence:
\begin{theorem}
\label{Thm:Lp}
For $p \in \left(0,-\frac{u_0}{\bar B}\ln E( W_1^\beta ) \right)$ we find
\[E\left(|| \pi_n^{Y_{1:n}}-\tilde \pi_n^{Y_{1:n}}||_{var}^p\right) =\mathcal O(\rho^n), \qquad n\to \infty,\]
where
$\rho=\left(E ( W_1^\beta ) e^{\frac{p \bar B}{u_0}} \right)^{\frac p{p+\beta}}$ and $\bar B=\frac{o_0-u_0}2 \left(
\beta(\beta+1)
-\frac{2\beta^2 u_0}{o_0}
+\frac{\beta(\beta+1)o_0^2}{u_0^2}
\right)^\frac12
$.
\end{theorem}

Note that the theorem achieves lower rates than those in the pure Gamma case since
\begin{equation*}
\left(E ( W_1^\beta ) e^{\frac{p \bar B}{u_0}} \right)^{\frac p{p+\beta}}
\;\ge\;
E ( W_1^\beta )^{\frac p{p+\beta}}
\;\ge\;
E ( W_1^{\frac {p\beta}{p+\beta}})
\;\ge\;
E ( W_1^p ).
\end{equation*}

The proof is based on Theorem \ref{Thm:stability} and the bound \eqref{Jestimate} for the total variation distance. Consider first the following Lemma specifying the behavior of the random index $J_\delta$ introduced in the proof of Thm.\ \ref{Thm:stability}.

\begin{lemma}
\label{Lem:J}
Let $0<\delta<E(W_1^\beta)^{-\frac1\beta}$.
For
\[J_\delta \;=\; \inf\{k \in \N : b^jY_j>\delta^j \text{ for all } j \ge k\} \]
we find
\begin{equation}
\label{Eq:J}
P(J_\delta > n) \;\le\; \frac{E\left(X_0^\beta\right)}{\Gamma(\beta+1)(1-q)} q^{n}, \qquad n \in \N_0,
\end{equation}
where $q=\delta^\beta E (W_1^\beta)$.
\end{lemma}

\begin{proof}
Using \eqref{Eq:bound} we have for $n \in \N_0$
\begin{eqnarray*}
P(J_\delta > n)
\;=\;
P\left( \cup_{j=n}^\infty \{b^jY_j \le \delta^j\} \right)
&\le&
\sum_{j=n}^\infty P\left(b^jY_j \le \delta^j\right)\\
&\le&
\frac{E\left(X_0^\beta\right)}{\Gamma(\beta+1)} \cdot \frac{q^{n}}{1-q}.
\end{eqnarray*}
\end{proof}

\begin{proof} of Thm.\ \ref{Thm:Lp}:
First note that for $s \ge 0$
\[ E\left(e^{sJ_\delta}\right)=1+(e^s-1)\sum_{n=0}^\infty e^{sn}P(J_\delta> n).\]

Due to Lemma \ref{Lem:J} the expectation is finite if $e^s q<1$, where $q=\delta^\beta E (W_1^\beta)$.

Now let
$\delta=\left(E ( W_1^\beta ) e^{\frac{p \bar B}{u_0}} \right)^{-\frac1{p+\beta}}$.
Then $\delta$ satisfies
$1<\delta<E(W_1^\beta)^{-\frac1\beta}$. Now, Lemma \ref{Lem:J} applies and we obtain from \eqref{Jestimate} that
\begin{eqnarray*}
\lefteqn{E\left(|| \pi_n^{Y_{1:n}}-\tilde \pi_n^{Y_{1:n}}||_{var}^p\right)}\\
&\le&
Q ||h_0||_{Lip}
E \left(e^{p \bar B \frac\delta{\delta-1}} \frac{e^{\frac{p\bar B}{u_0}J_\delta}}{\delta^{pn}} \I_{\{J_\delta\le n\}}
+
e^{\frac{p\bar B}{u_0}n}\I_{\{J_\delta> n\}}
 \right)\\
&\le&
Q ||h_0||_{Lip} \left(
e^{p \bar B \frac\delta{\delta-1}} \frac{E\left(e^{\frac{p\bar B}{u_0}J_\delta}\right)}{\delta^{pn}}
+
e^{\frac{p\bar B}{u_0}n}P(J_\delta > n)
\right)\\
&\le&
Q ||h_0||_{Lip} \left(
e^{p \bar B \frac\delta{\delta-1}} \frac{E\left(e^{\frac{p\bar B}{u_0}J_\delta}\right)}{\delta^{pn}}
+
\frac{E\left(X_0^\beta\right)}{\Gamma(\beta+1)} e^{\frac{p\bar B}{u_0}n}q^n
\right),
\end{eqnarray*}
which tends to 0 as $\mathcal O(\rho^n)$ since $e^{\frac{p\bar B}{u_0}}q=\delta^{-p}=\rho$ and $\rho<1$ due to the assumptions on $p$.
\end{proof}

Note that in this proof the choice $\delta=\left(E ( W_1^\beta ) e^{\frac{p \bar B}{u_0}} \right)^{-\frac1{p+\beta}}$ is optimal as the two components of the above bound are directed opposite to one another and the actual $\delta$ is chosen such that their behavior merges.

\bibliographystyle{plain}
\bibliography{Literatur_Stability}

\end{document}